\documentclass[10pt,letterpaper,twoside]{article}

\usepackage[letterpaper,left=1in,right=1in,top=1.5in,bottom=1.5in]{geometry}

\title{The animation of the opposite of finite sets}
\author{Benjamin Antieau}
\date{\today}

\usepackage[pdfstartview=FitH,
            pdfauthor={},
            pdftitle={},
            colorlinks,
            linkcolor=reference,
            citecolor=citation,
            urlcolor=e-mail,
            backref]{hyperref}

\usepackage{amsmath}
\usepackage{amscd}
\usepackage{amsbsy}
\usepackage{amssymb}
\usepackage{verbatim}
\usepackage{eufrak}
\usepackage{eucal}
\usepackage{microtype}
\usepackage{hyperref}
\usepackage{mathrsfs}
\usepackage{amsthm}
\usepackage{stmaryrd}
\usepackage{xy}
\input xy
\xyoption{all}
\usepackage{tikz}
\usetikzlibrary{matrix,arrows}
\usepackage{tikz-cd}
\usepackage{dsfont}
\usepackage[T1]{fontenc}
\usepackage{enumitem}
\setlist{noitemsep}

\usepackage{fancyhdr}
\usepackage[bbgreekl]{mathbbol}
\usepackage{amsfonts}
\DeclareSymbolFontAlphabet{\mathbb}{AMSb} 
\DeclareSymbolFontAlphabet{\mathbbl}{bbold}

\usepackage{yhmath}

\usepackage{color}
\definecolor{todo}{rgb}{1,0,0}
\definecolor{conditional}{rgb}{0,1,0}
\definecolor{e-mail}{rgb}{0,.40,.80}
\definecolor{reference}{rgb}{.20,.60,.22}
\definecolor{mrnumber}{rgb}{.80,.40,0}
\definecolor{citation}{rgb}{0,.40,.80}

\let\oldmarginpar\marginpar
\renewcommand\marginpar[1]{\-\oldmarginpar[\raggedleft\footnotesize #1]%
{\raggedright\footnotesize #1}}

\newcommand{\Ascr}{\mathcal{A}}

\newcommand{\Cscr}{\mathcal{C}}
\newcommand{\Dscr}{\mathcal{D}}

\newcommand{\Fscr}{\mathcal{F}}

\newcommand{\Pscr}{\mathcal{P}}

\newcommand{\Sscr}{\mathcal{S}}

\newcommand{\B}{\mathrm{B}}

\newcommand{\C}{\mathrm{C}}

\newcommand{\D}{\mathrm{D}}

\renewcommand{\L}{\mathrm{L}}

\newcommand{\bE}{\mathbf{E}}
\newcommand{\bF}{\mathbf{F}}

\newcommand{\bP}{\mathbf{P}}
\newcommand{\bQ}{\mathbf{Q}}

\newcommand{\bZ}{\mathbf{Z}}

\newcommand{\Fin}{\Fscr\mathrm{in}}

\newcommand{\Ab}{\Ascr\mathrm{b}}

\newcommand{\op}{\mathrm{op}}

\newcommand{\Mod}{\mathrm{Mod}}

\newcommand{\Ind}{\mathrm{Ind}}
\newcommand{\Pro}{\mathrm{Pro}}

\newcommand{\CAlg}{\mathrm{CAlg}}

\newcommand{\DAlg}{\mathrm{DAlg}}

\newcommand{\LSym}{\mathrm{LSym}}

\newcommand{\Cont}{\mathrm{Cont}}
\newcommand{\cts}{\mathrm{cts}}
\newcommand{\Bool}{\mathrm{Bool}}

\newcommand{\sfp}{\mathrm{sfp}}
\newcommand{\An}{\mathrm{An}}

\newcommand{\Set}{\mathrm{Set}}

\newcommand{\fin}{\mathrm{fin}}

\newcommand{\perf}{\mathrm{perf}}

\renewcommand{\geq}{\geqslant}
\renewcommand{\leq}{\leqslant}

\newcommand{\stackspace}{2.5}
\newcommand{\stack}[2][1cm]{\;\tikz[baseline, yshift=.65ex]%
    {\foreach \k [evaluate=\k as \r using (.5*#2+.5-\k)*\stackspace] in {1,...,#2}{%
    \ifodd\k{\draw[->](0,\r pt)--(#1,\r pt);}%
    \else{\draw[<-](0,\r pt)--(#1,\r pt);}\fi
    }}\;}

\newcommand{\Map}{\mathrm{Map}}

\newcommand{\Hom}{\mathrm{Hom}}
\newcommand{\Fun}{\mathrm{Fun}}

\DeclareMathOperator*{\colim}{colim}

\DeclareMathOperator{\Spec}{Spec}

\newcommand{\we}{\simeq}
\newcommand{\iso}{\cong}

\theoremstyle{plain}
\newtheorem{theorem}{Theorem}[section]
\newtheorem*{theorem*}{Theorem}
\newtheorem{lemma}[theorem]{Lemma}

\newtheorem{corollary}[theorem]{Corollary}
\newtheorem*{corollary*}{Corollary}

\theoremstyle{plain}

\theoremstyle{definition}

\newtheoremstyle{named}{}{}{\itshape}{}{\bfseries}{.}{.5em}{#1 \thmnote{#3}}
\theoremstyle{named}

\theoremstyle{definition}
\newtheorem{definition}[theorem]{Definition}

\newtheorem{notation}[theorem]{Notation}

\newtheorem{example}[theorem]{Example}
\newtheorem*{example*}{Example}

\newtheorem*{question*}{Question}
\newtheorem{construction}[theorem]{Construction}

\newtheorem{remark}[theorem]{Remark}

\begin{document}

\maketitle

\begin{abstract}
    \noindent
    We give a direct proof of the fact that the animation of the category $\Fin^\op$ is a $1$-category.
\end{abstract}

\section{Introduction}

Various categories of commutative rings arise in the attempt to characterize homotopy types. Over
$\bQ$, this goes back to Sullivan~\cite{sullivan-infinitesimal} who used commutative differential
graded $\bQ$-algebras to study rational
homotopy theory. Over $\bF_p$, Mandell's theorem~\cite{mandell-cochains} captures $p$-adic homotopy-theoretic information
via the $\bF_p$-cochains $\C^*(X;\bF_p)$ as an $\bE_\infty$-algebra equipped with a trivialization of the power
operation $\Pscr^0$, which serves as Frobenius. More recently, several
authors~\cite{antieau_spherical,ekedahl-minimal,horel,ksz} have worked with
derived notions of binomial rings to study homotopy types of finite type nilpotent homotopy types over
$\bZ$.

The characteristic $p$ analogue of a binomial ring is a $p$-Boolean ring, a commutative
$\bF_p$-algebra $R$ such that $x^p=x$ for all $x\in R$. These are precisely those $\bF_p$-algebras
whose Frobenius is the identity.

The categories $\CAlg_{\bF_p}^{\varphi=1}$ of $p$-Boolean rings are all equivalent to the category
$\Ind(\Fin^\op)\we\Pro(\Fin)^\op$ via Stone duality. Indeed, one can check that
the rings of functions $\bF_p^S$ where $S$ is a finite set form a set of compact
generators of $\CAlg_{\bF_p}^{\varphi=1}$. (Stone duality is typically stated for $p=2$; this version
goes back to~\cite{stringall}.)

The notion of $p$-Boolean rings can be extended to the $\infty$-category of derived commutative
$\bF_p$-algebras in the sense of~\cite{raksit}. The associated $\infty$-category
$\DAlg_{\bF_p}^{\varphi=1}$ is equivalent by~\cite[Thm.~4.17]{antieau_spherical} to the
$\infty$-category $\Pro(\Sscr_{p\fin})^\op$, the opposite of
$\infty$-category of $p$-adic homotopy types. (See also~\cite{kriz}.) The composition
$\Pro(\Sscr_{p\fin})^\op\we\DAlg_{\bF_p}^{\varphi=1}\rightarrow\DAlg_{\bF_p}$ is given by taking
continuous cochains.

In particular, every object $R$ of $\DAlg_{\bF_p}^{\varphi=1}$ is coconnective in the sense that
$\pi_iR=0$ for $i>0$. It follows that the animation of the category $\CAlg_{\bF_p}^{\varphi=1}$ is
a $1$-category. However, as we remarked above, $\CAlg_{\bF_p}^{\varphi=1}=\Ind(\Fin^\op)$, so this
result can be interpreted as stating that the animation of $\Fin^\op$ is a $1$-category.
Our goal in this note is to review the algebraic proof of this fact, which effectively goes back
to~\cite[Prop.~11.6]{bhatt-scholze-witt},
and to give a direct $\infty$-categorical proof of this statement that does 
not resort to the theory of animated commutative rings.

\begin{theorem}\label{thm:intro}
    The animation $\An(\Fin^\op)$ is a $1$-category.
\end{theorem}

This is one of only two examples of small categories with finite coproducts whose animations are
discrete of which we are aware. The other is the $\infty$-category of animated perfect commutative
$\bF_p$-algebras, which was proved in
Bhatt--Scholze~\cite[Prop.~11.6]{bhatt-scholze-witt} and is an input into the algebraic proof of
Theorem~\ref{thm:intro}.

As a corollary, we deduce the following result.

\begin{corollary}\label{cor:intro}
    Let $X^\bullet$ be a cosimplicial finite set with limit $X^{-1}$. Let $\bP(X^\bullet)$ be its power set, which is a
    simplicial finite set. Then, the natural map $\bP(X^\bullet)\rightarrow\bP(X^{-1})$ is a colimit
    diagram in $\Sscr$.
\end{corollary}

The corollary says in particular that $\pi_i|\bP(X^\bullet)|=0$ for $i\geq 1$ at any basepoint.
As the animation of $\Fin$ is the $\infty$-category of anima, or homotopy types, itself, we view
Theorem~\ref{thm:intro} as an obstruction to finding a meaningful notion of `coconnective' homotopy
type.

As $\Ind(\Fin^\op)\we\Pro(\Fin)^\op$ is equivalent to the category $\Bool$ of Boolean algebras,
Theorem~\ref{thm:intro} implies that $\An\Bool\we\Bool$: an animated Boolean
algebra is discrete. See Section~\ref{sec:boolean}.

{\em Added 1 August 2025.} Since the first version of this paper was submitted, another argument
for the proof of Theorem~\ref{thm:intro} was given by Lehner in~\cite[Thm.~5.6]{lehner}.

\paragraph{Conventions.} A space is a topological space. We use profinite sets and profinite spaces
interchangeably. An anima, or homotopy type, is an object of the $\infty$-category $\Sscr$ given as
the homotopy coherent nerve of the simplicial category of Kan complexes.

An anima $X\in\Sscr$ is discrete if for each base point $x\in X$ the
homotopy groups $\pi_i(X,x)$ vanish. The inclusion $\Set\rightarrow\Sscr$ is fully faithful with
essential image the full subcategory of discrete anima. An $\infty$-category $\Cscr$ is a
$1$-category if the mapping anima $\Map_\Cscr(X,Y)$ is discrete for each $X,Y\in\Cscr$.
We use $1$-category and category interchangeably.

\paragraph{Acknowledgments.}
We thank Lukas Brantner and Gijs Heuts for their organization of an enjoyable and productive
workshop in Utrecht in 2023. We would also like to thank Thomas Nikolaus for encouragement in finding a direct proof and Peter
Haine for related discussions. This paper was significantly improved thanks to comments by two
referees, whom we also thank.

We were supported by NSF grant DMS-2152235, Simons Fellowship 00005925, and the Simons
Collaboration on Perfection. This work was completed during a visit to the Max
Planck Institute for Mathematics in Bonn and written up while
on leave at UC Berkeley.

\section{Animation}

The notion of animation goes back to Quillen~\cite{quillen_homotopical} who studied simplicial model category structures on
categories of simplicial objects in various settings, such as simplicial abelian groups or
simplicial commutative rings. A detailed $\infty$-categorical treatment is given by Lurie
in~\cite[Sec.~5.5.8]{htt}. The term `animation' was introduced by \v{C}esnavi\v cius--Scholze
in~\cite{cesnavicius-scholze}.

\begin{definition}[Siftedness]
    A simplicial set $X$ is sifted if it is nonempty and if the diagonal map $\Delta_X\colon X\rightarrow X\times
    X$ is cofinal in the sense of~\cite[Def.~4.1.1.1]{htt}.
    Cofinality of $\Delta_X$ is equivalent to the requirement that for every $\infty$-category $\Cscr$ and every
    functor $F\colon X\times X\rightarrow\Cscr$ the
    colimit of $F$ can be computed by its restriction to $X$ in the sense that the natural map
    $\colim(F\circ\Delta_X)\rightarrow\colim(F)$ is an equivalence, assuming that $\colim(F)$ exists.
    See~\cite[Prop.~4.1.1.8]{htt}.
\end{definition}

\begin{example}
    The following examples are found in~\cite[Sec.~5.5.8]{htt}.
    \begin{enumerate}
        \item[(a)] A filtered $\infty$-category is sifted.
        \item[(b)] The category $\Delta^\op$ is sifted.
    \end{enumerate}
    These are in some sense the two most important examples in practice. Let 
    $F\colon\Cscr\rightarrow\Dscr$ be a functor and suppose that $\Cscr$ has all small colimits.
    Then, $F$ preserves sifted colimits if and only if it preserves filtered colimits and geometric
    realizations (colimits along $\Delta^\op$). See~\cite[Cor.~5.5.8.17]{htt}.
\end{example}

\begin{definition}[$1$-siftedness]
    A category $X$ is $1$-sifted if $X$-shaped colimits in $\Set$ commute with finite products.
\end{definition}

\begin{example}
    See~\cite{adamek-rosicky-vitale} for the following examples.
    \begin{enumerate}
        \item[(a)] Filtered categories are $1$-sifted.
        \item[(b)] The category $\Delta^{\leq 1,\op}$ is $1$-sifted.
    \end{enumerate}
    Colimits over $\Delta^{\leq 1,\op}$ are called reflexive coequalizers. Let
    $F\colon\Cscr\rightarrow\Dscr$ be a functor between $1$-categories. If $\Cscr$ has small
    colimits, then $F$ preserves $1$-sifted colimits if and only if it preserves filtered colimits
    and reflexive coequalizers by~\cite[Thm.~2.1]{adamek-rosicky-vitale}.
\end{example}

\begin{definition}
    Let $\Cscr$ be a cocomplete $\infty$-category. Say that $X\in\Cscr$ is projective if the functor
    $\Map_\Cscr(X,-)\colon\Cscr\rightarrow\Sscr$ preserves geometric realizations. If $\Cscr$ is a
    cocomplete $1$-category, say that an object $X\in\Cscr$ is $1$-projective if
    $\Hom_\Cscr(X,-)\colon\Cscr\rightarrow\Set$ preserves reflexive coequalizers.
    If additionally $\Map_\Cscr(X,-)$ preserves filtered colimits, then $X$ is called compact
    projective. In this case, $\Map_\Cscr(X,-)$ preserves all sifted colimits. Similarly, for a
    $1$-category $\Cscr$, we have
    that $X$ is compact $1$-projective if $\Hom_\Cscr(X,-)\colon\Cscr\rightarrow\Set$ 
    preserves $1$-sifted colimits.
\end{definition}

\begin{notation}
    Let $\Cscr^\sfp\subseteq\Cscr$ denote the full subcategory of compact projective objects. If
    $\Cscr$ is a $1$-category, let $\Cscr^{1\sfp}\subseteq\Cscr$ denote the full subcategory of
    compact $1$-projective objects.
\end{notation}

\begin{remark}
    If $\Cscr$ is a $1$-category, then every projective object in $\Cscr$ is $1$-projective. The
    converse is not true because $\Set$ is not closed in $\Sscr$ under geometric realizations.
\end{remark}

\begin{example}
    In $\Set$ every finite set is compact $1$-projective. The only projective object of $\Set$ is
    $\emptyset$.
\end{example}

\begin{remark}
    If $\Cscr$ is a cocomplete $\infty$-category, then $\Cscr^\sfp$ is closed under finite coproducts.
    Indeed, $\Cscr^\sfp$ contains the initial object of $\Cscr$ and if $X,Y\in\Cscr^\sfp$ and $F\colon I\rightarrow\Cscr$ is a sifted diagram, then
    \begin{align*}
        \Map_\Cscr(X\amalg Y,\colim_{i\in I}F(i))&\we\Map_\Cscr(X,\colim_{i\in
        I}F(i))\times\Map_\Cscr(Y,\colim_{i\in
        I}F(i))\\
        &\we(\colim_{i\in I}\Map_\Cscr(X,F(i)))\times(\colim_{i\in I}\Map_\Cscr(Y,F(i)))\\
        &\we\colim_{i\in
        I}(\Map_\Cscr(X,F(i))\times\Map_\Cscr(Y,F(i)))\\
        &\we\colim_{i\in I}\Map_\Cscr(X\amalg Y,F(i)),
    \end{align*}
    where the third equivalence is by siftedness.
    Similarly, if $\Cscr$ is a $1$-category with $1$-sifted colimits, then $\Cscr^{1\sfp}$ is closed under finite coproducts.
\end{remark}

We will use animation in two closely related senses.

\begin{definition}[Animation of a small category]\label{def:an_small}
    Let $\Cscr$ be a small $\infty$-category which admits finite coproducts. Its animation is
    the full subcategory
    $\An(\Cscr)=\Fun^\sqcap(\Cscr^\op,\Sscr)\subseteq\Pscr(\Cscr)=\Fun(\Cscr^\op,\Sscr)$
    of presheaves which preserve finite products.
\end{definition}

By the results of~\cite[Sec.~5.5.8]{htt}, the animation of a small $\infty$-category $\Cscr$ with finite
coproducts can be viewed as the presentable $\infty$-category obtained by freely adjoining sifted
colimits to $\Cscr$. Specifically, restriction along the Yoneda embedding
$\Cscr\rightarrow\An(\Cscr)$ induces a functor
$\Fun(\An(\Cscr),\Dscr)\rightarrow\Fun(\Cscr,\Dscr)$. If $\Dscr$ admits sifted colimits, then
by~\cite[Prop.5.5.8.15]{htt} this functor
restricts to an equivalence
\begin{equation}\label{eq:fun_sigma}
    \Fun_\Sigma(\An(\Cscr),\Dscr)\we\Fun(\Cscr,\Dscr),
\end{equation}
where
$\Fun_\Sigma(\An(\Cscr),\Dscr)\subseteq\Fun(\An(\Cscr),\Dscr)$ is the full subcategory of functors
which preserve sifted colimits. Moreover, if $\Dscr$ admits all colimits, then the equivalence~\eqref{eq:fun_sigma} restricts to an
equivalence $\Fun^\L(\An(\Cscr),\Dscr)\we\Fun^\sqcup(\Cscr,\Dscr)$ between the $\infty$-categories
of colimit-preserving functors $\An(\Cscr)\rightarrow\Dscr$ and the finite coproduct-preserving
functors $\Cscr\rightarrow\Dscr$.

\begin{definition}
    Suppose that $\Cscr$ is a cocomplete $\infty$-category. Say that $\Cscr$ is compact
    projectively generated if the natural functor $\An(\Cscr^\sfp)\rightarrow\Cscr$ given by left Kan
    extension is an equivalence. Similarly, if $\Cscr$ is a cocomplete $1$-category, say that
    $\Cscr$ is compact $1$-projectively generated if the natural functor
    $\Fun^\sqcap(\Cscr^{1\sfp,\op},\Set)\rightarrow\Cscr$ is an equivalence.
\end{definition}

\begin{remark}
    Let $\Cscr$ be a small $\infty$-category with finite coproducts. Every object of $\Cscr$ is
    compact projective when viewed as an object of $\An(\Cscr)$ via the Yoneda embedding
    $\Cscr\subseteq\An(\Cscr)$; see~\cite[Prop.~5.5.8.22]{htt}.
\end{remark}

\begin{definition}[Animation of a large $1$-category]\label{def:an_large}
    Let $\Cscr$ be a $1$-category which is compact $1$-projectively generated.
    The animation of $\Cscr$ is defined to be $\An\Cscr=\An(\Cscr^{1\sfp})$.
\end{definition}

In the situation of Definition~\ref{def:an_large}, the category $\Cscr$ is equivalent to the full
subcategory of $\An\Cscr$ corresponding to finite product-preserving functors
$\Cscr^{1\sfp}\rightarrow\Sscr$ which take discrete values. The inclusion $\Cscr\subseteq\An\Cscr$
admits a left adjoint given by applying $\pi_0$ to the values of finite product-preserving functors.

\begin{example}
    Let $k$ be a commutative ring. The categories
    $\Set$, $\Mod_k$, and $\CAlg_k$ are compact $1$-projectively generated by the categories of
    finite sets, finitely presented projective $k$-modules, and retracts of finitely presented polynomial $k$-algebras
    $k[x_1,\ldots,x_n]$.
    \begin{enumerate}
        \item[(a)] The animation of $\Set$ is $\Sscr$. Finite sets are $1$-projective in $\Set$
            but projective when viewed as objects of $\Sscr$.
        \item[(b)] The animation of $\Mod_k$ is $\D(k)_{\geq 0}$, the connective part of the
            derived $\infty$-category of $k$. This is equivalent to the $\infty$-category
            associated to the simplicial model category of simplicial $k$-modules
            by~\cite[Cor.~5.5.9.3]{htt}.
            Projective $k$-modules are
            $1$-projective in $\Mod_k$ but projective in $\D(k)_{\geq 0}$.
        \item[(c)] The animation of $\CAlg_k$ is the category of animated commutative $k$-algebras,
            by fiat. This is equivalent to the $\infty$-category associated to the simplicial model
            category of simplicial commutative $k$-algebras by~\cite[Cor.~5.5.9.3]{htt}.
            Finitely presented polynomial $k$-algebras are $1$-projective in $\CAlg_k$ but
            projective in $\An\CAlg_k$.
    \end{enumerate}
\end{example}

In the examples above, the only projective objects in $\Set$, $\Mod_k$, and $\CAlg_k$ are the
initial objects. Few projective objects in $1$-categories are known in general.

\begin{lemma}\label{lem:if_discrete}
    Suppose that $\Cscr$ is a $1$-category with all colimits and suppose that $X\in\Cscr$ is $1$-projective. If for all
    $Y_\bullet\colon\Delta^\op\rightarrow \Cscr$ the geometric realization
    $|\Hom_\Cscr(X,Y_\bullet)|$ is discrete, then $X$ is projective.
\end{lemma}

\begin{proof}
    Let $Y_{-1}=|Y_\bullet|$. As $\Cscr$ is a $1$-category, $Y_{-1}$ is the reflexive coequalizer
    of $Y_1\stack{3}Y_0$.
    But, by assumption, $|\Hom_\Cscr(X,Y_\bullet)|\we\pi_0|\Hom_\Cscr(X,Y_\bullet)|$, and the
    latter is equivalent to the reflexive coequalizer of
    $\Hom_\Cscr(X,Y_1)\stack{3}\Hom_\Cscr(X,Y_0)$, which is $\Hom_\Cscr(X,Y_{-1})$, since we have
    assumed that $X$ is $1$-projective. 
\end{proof}

\begin{lemma}\label{lem:discrete_values}
    Let $\Cscr$ be a compact $1$-projectively generated $1$-category.
    The following conditions are equivalent:
    \begin{enumerate}
        \item[{\em (i)}] the natural functor
            $\pi_0\colon\Fun^\sqcap(\Cscr^{1\sfp,\op},\Sscr)\rightarrow\Fun^\sqcap(\Cscr^{1\sfp,\op},\Set)$ is an
            equivalence;
        \item[{\em (ii)}] the $\infty$-category $\An\Cscr$ is a $1$-category;
        \item[{\em (iii)}] the objects of $\Cscr^{1\sfp}$ are projective in $\Cscr$.
    \end{enumerate}
\end{lemma}

Condition (i) may be rephrased as the condition that $\pi_0\colon\An\Cscr\rightarrow\Cscr$ is an
equivalence.

\begin{proof}[Proof of Lemma~\ref{lem:discrete_values}]
    As $\Fun^\sqcap(\Cscr^{1\sfp,\op},\Set)$ is a $1$-category, condition (i) implies (ii).
    Suppose that $\An\Cscr$ is a $1$-category. 
    Since $\An\Cscr$ is a $1$-category, $\Map_{\An\Cscr}(X,Y)$ is discrete for every
    $X\in\Cscr^{1\sfp}$. In particular, by Lemma~\ref{lem:if_discrete}, $X$ is projective: (ii)
    implies (iii). Now, assume that every object $X$ of $\Cscr^{1\sfp}$ is projective.
    Every object $Y$ of $\An\Cscr$
    can be written as a geometric realization $Y\we|Z_\bullet|$ in $\Fun(\Cscr^{1\sfp,\op},\Sscr)$
    of a simplicial object where each $Z_n\in\Ind(\Cscr^{1\sfp})\subseteq\Cscr$. See~\cite[Lem.~5.5.8.14]{htt}.
    Now, discreteness of $\Map_{\An\Cscr}(X,Y)$ follows from projectivity of $X$. It follows that
    $\pi_0$ is an equivalence.
\end{proof}

\section{Boolean rings}\label{sec:boolean}

The category $\Pro(\Fin)^\op\we\Ind(\Fin^\op)$ is equivalent to the category
$\CAlg_{\bF_2}^{\varphi=1}$ of ($2$-)Boolean rings and
to the category $\Bool$ of Boolean algebras. The equivalence is implemented by
the adjoint equivalence
$$\Cont(-,\bF_p)\colon\Pro(\Fin)^\op\rightleftarrows\CAlg_{\bF_2}^{\varphi=1}\colon\Spec,$$
where $\Cont(-,\bF_p)$ is the continuous $\bF_2$-valued functions functor.

\begin{lemma}\label{lem:1projective}
    Suppose that $X$ is a nonempty finite set. Then, it is $1$-projective when viewed as an object of
    $\Ind(\Fin^\op)$.
\end{lemma}

\begin{proof}
    We use the incarnation of $\Ind(\Fin^\op)$ as $\CAlg_{\bF_2}^{\varphi=1}$. This category admits
    a conservative forgetful functor to $\CAlg_{\bF_2}$ which admits both a left and a right
    adjoint. In particular, the forgetful functor preserves all colimits, so the left adjoint
    $(-)_{\varphi=1}\colon\CAlg_{\bF_2}\rightarrow\CAlg_{\bF_2}^{\varphi=1}$ preserves compact $1$-projective objects.
    Since $\CAlg_{\bF_2}$ is compact $1$-projectively generated, it follows that
    $\CAlg_{\bF_2}^{\varphi=1}$ is compact $1$-projectively generated too. Given
    a compact $1$-projective generator
    $\bF_2[x_1,\ldots,x_n]$ of $\CAlg_{\bF_2}$ for some integer $n\geq 0$, we have
    $$(\bF_2[x_1,\ldots,x_n])_{\varphi=1}\iso\bF_2[x_1,\ldots,x_n]/(x_1^2-x_1,\ldots,x_n^2-x_n),$$
    which corresponds to the (opposite) finite set $\{0,1\}^n$. Thus, $\{0,1\}^n$ is $1$-projective
    in $\Ind(\Fin^\op)$. As every nonempty finite set is a retract of some $\{0,1\}^n$, every
    nonempty finite set is $1$-projective in $\Ind(\Fin^\op)$.
\end{proof}

\begin{remark}
    The empty set $\emptyset\in\Ind(\Fin^\op)$ is not $1$-projective.
\end{remark}

\begin{remark}\label{rem:an_bool}
    The animation $\An\Ind(\Fin^\op)$ is thus equivalent to $\An(\Fin^\op_{\geq 1})$, the animation
    of the opposite of the category of nonempty finite sets. The inclusion $\Fin^\op_{\geq
    1}\rightarrow\Fin^\op$ induces a functor $\An\Ind(\Fin^\op_{\geq 1})\rightarrow\An(\Fin^\op)$,
    which is fully faithful by~\cite[Prop.~5.5.8.22]{htt}. As the latter is a $1$-category by
    Theorem~\ref{thm:intro}, it follows that $\An\Ind(\Fin^\op)$ is a $1$-category and in
    particular that each nonempty set $X$ is projective in $\Ind(\Fin^\op)$ thanks to
    Lemma~\ref{lem:discrete_values}.
\end{remark}

\section{The algebraic argument}

We start by comparing the perfect animated $\bF_p$-algebras and the animation of the category of
perfect $\bF_p$-algebras.

\begin{construction}
    Recall that any commutative $\bF_p$-algebra $R$ admits a Frobenius endomorphism $\varphi\colon
    R\rightarrow R$ given by $\varphi(x)=x^p$. The Frobenius defines a natural endomorphism, also
    denoted $\varphi$, of the
    identity functor on $\CAlg_{\bF_p}$ and also of the identity functor on $\CAlg_{\bF_p}^{1\sfp}$.
    As $\An\CAlg_{\bF_p}\we\Fun^\sqcap(\CAlg_{\bF_p}^{1\sfp,\op},\Sscr)$, composition with
    $\varphi^\op$ induces a natural transformation $\varphi$ of the identity functor on $\An\CAlg_{\bF_p}$.
\end{construction}

\begin{definition}
    An animated commutative $\bF_p$-algebra $R$ is perfect if $\varphi\colon R\rightarrow R$ is an
    equivalence. Let $(\An\CAlg_{\bF_p})^{\perf}\subseteq\An\CAlg_{\bF_p}$ denote the full
    subcategory of perfect animated commutative $\bF_p$-algebras.
\end{definition}

On the other hand, we have the animated perfect $\bF_p$-algebras.

\begin{definition}
    The full subcategory $\CAlg_{\bF_p}^{\perf}\subseteq\CAlg_{\bF_p}$ is compact $1$-projectively
    generated. Compact $1$-projective generators are given by the perfections
    $\bF_p[x_1^{1/p^\infty},\ldots,x_n^{1/p^\infty}]$ of the finitely presented polynomial
    $\bF_p$-algebras. The $\infty$-category of animated perfect $\bF_p$-algebras is
    $\An\CAlg_{\bF_p}^\perf$.
\end{definition}

As on $\An(\CAlg_{\bF_p})$ there is a Frobenius endomorphism of the identity functor of
$\An(\CAlg_{\bF_p}^\perf)$ which is in fact a natural automorphism.

There is a functor
$$\An\CAlg_{\bF_p}^\perf\rightarrow\An\CAlg_{\bF_p}$$
obtained by left Kan extension of the inclusion
$\CAlg_{\bF_p}^{\perf,\sfp}\rightarrow\CAlg_{\bF_p}\subseteq\An\CAlg_{\bF_p}$.

The following lemma was proved in the course of the proof of~\cite[Lem.~2.20]{antieau_spherical}.
It says that animated perfect commutative $\bF_p$-algebras are equivalent to perfect animated
commutative $\bF_p$-algebras.

\begin{lemma}
    The functor $\An\CAlg_{\bF_p}^\perf\rightarrow\An\CAlg_{\bF_p}$ is fully faithful with
    essential image $(\An\CAlg_{\bF_p})^\perf$.
\end{lemma}

The following was established in~\cite[Prop.~11.6]{bhatt-scholze-witt}.

\begin{lemma}\label{lem:animated_perfect_discrete}
    If $R\in(\An\CAlg_{\bF_p})^\perf$, then its underlying object in the derived $\infty$-category $\D(\bF_p)$ is discrete:
    $\pi_iR=0$ for $i\neq 0$.
\end{lemma}

\begin{proof}[Sketch of proof]
    It suffices to prove that the Frobenius induces the $0$ map on $\pi_1$ of
    $\LSym_{\bF_p}(\bF_p[1])$. Then, its perfection is $\bF_p$. It follows by taking Bar
    constructions that the perfections of
    all $\LSym_{\bF_p}(\bF_p[n])$ are equivalent to $\bF_p$ for $n\geq 1$.
\end{proof}

\begin{corollary}
    If $R=\bF_p[x_1,\ldots,x_n]$ for some integer $n\geq 0$, then the perfection $R_\perf$ of $R$
    is a projective object of $\CAlg_{\bF_p}^\perf$. Thus, $\An\CAlg_{\bF_p}^\perf$ is a $1$-category.
\end{corollary}

\begin{proof}
    By Lemma~\ref{lem:discrete_values}, it is enough to see that for compact projective objects
    $R_\perf$ of $\An\CAlg_{\bF_p}^\perf$ we have that $\Map_{\An\CAlg_{\bF_p}^\perf}(R_\perf,-)$
    take discrete values. It is enough to handle the case $R=\bF_p[x]$. Then,
    $\pi_i\Map_{\An\CAlg_{\bF_p}^\perf}(R_\perf,S)\iso\pi_i S$. We have that $\pi_iS=0$ for $i\neq
    0$ by Lemma~\ref{lem:animated_perfect_discrete}.
\end{proof}

We need the following lemma to handle a boundary case involving $\emptyset$ in both of our proofs
of Theorem~\ref{thm:intro}.

\begin{lemma}\label{lem:null_set}
    Consider the full subcategory $\Cscr$ of $\Fun(\Delta^0,\Sscr)\we\Sscr$ on those functors which
    preserve finite nonempty products. Then, $\Cscr\we\{\emptyset,\ast\}\subseteq\Sscr$.
\end{lemma}

\begin{proof}
    Let $0\in\Delta^0$. Then, we have $0\times 0\iso 0$ via either projection map. It follows that
    if $X\in\Cscr$, then $X\times X\we X$ via either projection map. By taking $\pi_0$, it follows
    that $X$ is empty or connected. Then, fixing a base point $x\in X$ and an integer $n\geq 1$, we obtain
    $\pi_n(X,x)\times\pi_n(X,x)\iso\pi_n(X,x)$ via either projection. But, the kernel of this map
    is $\pi_n(X,x)$, which must thus be zero for all $n\geq 1$. Hence, $X$ is either empty or is equivalent
    to $\ast$.
\end{proof}

\begin{proof}[Algebraic proof of Theorem~\ref{thm:intro}]
    By Lemma~\ref{lem:discrete_values}, it is enough to show that any functor
    $X\colon\Fin\rightarrow\Sscr$ which preserves finite products takes discrete values.
    There is a functor $\CAlg_{\bF_p}^{\perf,\sfp,\op}\rightarrow\Fin$ given by taking $R$ to
    $\Spec(R_{\varphi=1})$. On $\bF_p[x_1^{1/p^\infty},\ldots,x_n^{1/p^\infty}]$ this is
    $\Spec\bF_p[x_1,\ldots,x_n]/(x_1^p-x_1,\ldots,x_n^p-x_n)$, which is bijective to
    $\{0,\ldots,p-1\}^{n}$. This functor preserves products. It follows from
    Lemma~\ref{lem:animated_perfect_discrete} that the composition
    $\CAlg_{\bF_p}^{\perf,\sfp,\op}\rightarrow\Fin\xrightarrow{X}\Sscr$ takes discrete values.
    In particular, $X(S)$ is discrete for any finite $S$ which is a retract of $\{0,\ldots,p-1\}^n$ for
    some $n\geq 1$, in other words for any finite nonempty set $S$.
    To see that $X(\emptyset)$ is discrete, we can
    use that the inclusion $\Delta^0\we\{\emptyset\}\subseteq\Fin$ is
    closed under nonempty finite products so that the value of $X$ at $\emptyset$ must be
    $\emptyset$ or $\ast$ by Lemma~\ref{lem:null_set}.
\end{proof}

\section{The intrinsic argument}

\begin{proof}[Proof of Theorem~\ref{thm:intro}]
    Consider a product preserving functor $X\colon\Fin\rightarrow\Sscr$. It suffices by
    Lemma~\ref{lem:discrete_values} to prove that
    $X$ takes discrete values. As $X$ is a functor, it preserves retracts. Thus, if $S$ is a
    retract of $T^n$, then $X(S)$ is a retract of $X(T^n)\we X(T)^n$.

    As $X$ preserves finite products, it preserves grouplike $\bE_\infty$-objects. In particular,
    if a prime $p$ is fixed,
    as $\bF_p$ is the structure of a grouplike $\bE_\infty$-object on the set $\{0,\ldots,p-1\}$, the
    value $X(\{0,\ldots,p-1\})$ admits the structure of a grouplike $\bE_\infty$-object in $\Sscr$
    on which $p\we 0$. It follows that the higher homotopy groups of $X(\{0,\ldots,p-1\})$ at any
    basepoint are annihilated by $p$. If $\ell\neq p$ is another prime, then $\{0,\ldots,p-1\}$ is a
    retract, as a set, of $\{0,\ldots,\ell-1\}^n$ for a suitable integer $n\geq 0$. Thus, the homotopy
    groups of $X(\{0,\ldots,p-1\})$ are also annihilated by $\ell$. As $(p,\ell)=1$, we see that
    $\pi_i(X(\{0,\ldots,p-1\}))=0$ for $i\geq 1$ at any basepoint.
    For any nonempty set $S$, $X(S)$ is a retract of $X(\{0,\ldots,p-1\})^m$ for a suitable $m\geq
    1$. Thus, $X(S)$ is discrete if $S$ is nonempty. For the empty set, we argue as
    in the algebraic proof using Lemma~\ref{lem:null_set}.
\end{proof}

\begin{proof}[Proof of Corollary~\ref{cor:intro}]
    It follows from Theorem~\ref{thm:intro}, Lemma~\ref{lem:discrete_values}, and
    Remark~\ref{rem:an_bool}
    that in $\Ind(\Fin^\op)$ the object $\{0,1\}$ is projective and not merely
    $1$-projective. The corollary follows as $\{0,1\}$ corepresents the power set functor on
    $\Fin^\op$.
\end{proof}

\begin{remark}[Continuous power sets]
    More generally, suppose that $X^\bullet$ is a cosimplicial profinite space. Let
    $\bP_{\cts}(X^\bullet)$ denote the associated simplicial set. Then, $|\bP_\cts(X^\bullet)|$ is
    discrete. Here, if $Y$ is a profinite space, then its continuous power set $\bP_\cts(Y)$ is
    defined as $\colim_{i\in I}\bP(Y_i)$, where $Y\iso\lim_{i\in I}Y_i$ is a cofiltered limit of
    finite sets presenting $Y$. (Thus, $\bP_\cts(Y)$ is the Boolean algebra of clopen subsets of
    $Y$.)
\end{remark}

\begin{remark}
    It appears to be easier to prove Corollary~\ref{cor:intro} via the implication (i) implies
    (iii) of Lemma~\ref{lem:discrete_values} than checking it directly.
\end{remark}

\begin{remark}
    The animation of $\Fin^\op$ is philosophically related to the opposite of some homotopy theory of
    cosimplicial sets. The problem is: what notion of weak equivalence does one put on this
    category of cosimplicial sets? The only natural ones we know of are given by taking cohomology where one
    then obtains various notions of completed anima depending on the cohomology theory. 
\end{remark}

\begin{remark}
    The central feature of the proof of Theorem~\ref{thm:intro} is a sufficient supply of internal
    abelian group objects of various characteristics.
\end{remark}

\begin{remark}
    Let $\Ab^\fin$ denote the abelian category of finite groups. It admits a functor
    $\Ab^\fin\rightarrow\Fin$ which forgets the group structure. However, a typical finite
    product-preserving functor $\Ab^\fin\rightarrow\Sscr$ does not factor through $\Fin$ and in
    particular need not take on discrete values. An example is the classifying anima functor
    sending a finite abelian group $A$ to $\B A$. We see that the animation of $\Ab^{\fin,\op}$ is
    not a $1$-category.
\end{remark}

\small
\bibliographystyle{amsplain}
\bibliography{finop}

\medskip
\noindent
\textsc{Department of Mathematics, Northwestern University}\\
{\ttfamily antieau@northwestern.edu}

\end{document}